\documentclass[12pt]{amsart}
 \usepackage{amsfonts,amssymb}

\usepackage{mathrsfs}
\let\mathcal\mathscr

\usepackage{array}

\usepackage[all,ps,cmtip]{xy} 

\def\Z{{\bf Z}}

\def\OP2{\mathbb{OP}^2}
\def\CC{{\bf C}}
\def\PP{{\mathbb P}}

\def\ZZ{{\bf Z}}

\def\phi{{\varphi}}

\def\cX{\mathcal{X}}

\def\llra{\hbox to 10mm{\rightarrowfill}}
\def\lllra{\hbox to 15mm{\rightarrowfill}}

\def\llla{\hbox to 10mm{\leftarrowfill}}
\def\lllla{\hbox to 15mm{\leftarrowfill}}

\def\llra{\hbox to 10mm{\rightarrowfill}}
\def\lllra{\hbox to 15mm{\rightarrowfill}}

\newtheorem{lemm}{Lemma}[section]
\newtheorem{theo}[lemm]{Theorem}

\newtheorem{prop}[lemm]{Proposition}

\theoremstyle{definition}

\newtheorem{rema}[lemm]{Remark}

\theoremstyle{remark}
\newtheorem*{remark*}{Remark}
\newtheorem*{note*}{Note}

\begin{document}

\title{EPW sextics and Hilbert squares of K3 surfaces}

\author{Atanas Iliev \and Carlo Madonna}
\thanks{Research partially supported by the grant 
MTM2012-32670 of the Spanish Ministry of Economy and Competitiveness, 
by the Research fund for 
international Professors of the Seoul National University.} 

\begin{abstract}
We prove that the Hilbert square $S^{[2]}$ 
of a very general primitively polarized K3 surface 
$S$ of degree $d(n) = 2(4n^2 + 8n + 5)$, $n \ge 1$ 
is birational to a double Eisenbud-Popescu-Walter 
sextic. 
Our result implies a positive answers, in the case when $r$ is even, 
to a conjecture of O'Grady:  
On the Hilbert square of a 
very general K3 surface of genus $r^2+2$, $r \ge 1$ there is an antisymplectic 
involution. 
We explicitly give this involution on $S^{[2]}$ in term of the corresponding EPW polarization on it.

\end{abstract}

\maketitle

\section{Introduction and motivations}

O'Grady conjectured in \cite{OG1} that on the Hilbert square of a K3 surface 
of genus $g = r^2 + 2, r \ge 0$ there exists an antisymplectic involution 
(see (4.3.3) in \cite{OG1}).

We show here the following theorem, that in particular implies that
O'Grady conjecture is true
in the case when $r$ is even.

\begin{theo}
The Hilbert square of a very general $K3$ surface of degree $d(n)=2(4n^2 +8n+5)$, $n \geq 1$ is birational to a double $EPW$ sextic.
\end{theo}

Indeed, for $d(n) = 8n^2 + 16 n + 10 = 2(4(n+1)^2+1)$, 
the genus of $S$ is $g(n)$ = $d(n)/2 + 1 = 4n^2 + 8n + 6$ = $(2n+2)^2 + 2$, 
and for $n \ge 1$, \ $r = 2n+2$ covers all even numbers $r \ge 4$,
and the antisymplectic involution is determined in term of the EPW polarization
(see Section \ref{X4}).

Notice that while the
case $r = 0$ is well known, and the case $r = 2$ is studied in detail by O'Grady
(see e.g. \S 4.3 in \cite{OG1}),
in the cases of odd $r$ very little is known: only the case 
$r = 1$ is studied in \cite{De} 
and \cite{Fe}.

To show our main result we will follow O'Grady study of the case $r=2$, considering a double EPW sextic associated to a special K3 surface of degree $10$, together with
the methods used by Hassett in \cite{Ha1}.

\

The proof of our main theorem is given in Section \ref{s:main} while notations and 
basic facts and properties of double EPW sextic are recalled in Section \ref{s:preparatory}. 


\section{Fano fourfolds $X_{10}$, EPW sextics
and K3 surfaces}
\label{s:preparatory}

\subsection{Fano fourfolds $X_{10}$}
By $X_{10}$ we denote a prime Fano fourfold of index two 
and degree 10. 
By \cite{Mu1} and \cite{Gus}, any smooth $X_{10}$ is 
either a complete intersection 
of the Grassmannian $G = G(2,5) \subset \PP^9$ 
with a hyperplane and a quadric (the 1-st type) or a double covering 
of the smooth Fano fourfold $W_5 = G(2,5)\cap \PP^7$ 
branched along a quadratic section of $W_5$ 
(the 2-nd, or the Gushel' type).  
The moduli stack ${\mathcal X}_{10}$ of smooth 
$X_{10}$ is of dimension 24, and in ${\mathcal X}_{10}$ 
the general $X_{10}$ is from the first type. 
The condition for $X_{10}$ to be of the second type 
is of codimension 2, and 
the general $X_{10}$ of the second type is a smooth 
deformation from $X_{10}$ of the first type.

Let $X$ be a  fourfold of type $\cX_{10}$. By the Hodge-Riemann bilinear relations, 
the Hodge structure on $H^4(X,\Z)$ has weight 2, and the intersection form on $X$ 
endows the 4-th integral cohomology of $X$ with a structure of the lattice 
$\Lambda = H^4(X ,\Z) = I_{22,2}$,
where $I_{22,2}$ denotes the lattice $22\langle1\rangle \oplus 2\langle-1\rangle$.
 
By \cite{DIM1}, the lattice $H^4(X ,\Z)$ contains the fixed rank two polarization sublattice  
 $\Lambda_2:=H^4(G,\Z)\vert_X$ spanned on the restrictions to $X$ of the 
 two Schubert cycles $\sigma_{1,1}$ and $\sigma_2$ on $G = G(2,5)$. 
 In the basis 
$(u,v) = (\sigma_{1,1}\vert_X,\sigma_2\vert_X-\sigma_{1,1}\vert_X)$,   
 the intersection form of the lattice 
 $$
 \Lambda_2 = H^4(G,\Z)\vert_X = \Z u + \Z v
 $$
is given by 
$$
u^2 = v^2 = 2, \ \ uv = 0.
$$     

For $X = X_{10}$, 
the primitive cohomology lattice with respect to the lattice polarization $\Lambda_2$,
or the {\em vanishing cohomology lattice} is 
$$
\Lambda_0 = H^4(X ,\Z)_{\rm van}=\Lambda_2^\bot = 2E_8\oplus 2U\oplus 2\langle 2 \rangle ,
$$
ibid. 
$\Lambda_0$ is even of signature $(20,2)$. 

  



\subsection{EPW sextics}

Eisenbud-Popescu-Walter sextics, or in short EPW sextics,  
are special hypersurfaces of degree six in $\PP^5$, first introduced in
\cite{EPW} as examples of Lagrangian degeneracy loci. These hypersurfaces are
singular in codimension two, but O'Grady realized in \cite{OG1} \cite{OG2}
that they admit smooth double covers which are irreducible holomorphic
symplectic fourfolds. We will refer to this double covering as {\it double EPW sextic}.
In fact, the first examples of such double covers were
discovered by Mukai in \cite{Mu1}, who constructed them as moduli spaces of
stable rank two vector bundles on a polarized K3 surface of degree $10$.

Moreover O'Grady showed in \cite{OG2} that the generic such double cover is a
deformation of the Hilbert square of a K3 and that the family of double EPW sextics is a locally versal family of projective deformations of such an Hilbert square of a $K3$ surface.


Let $V$ be a $6$-dimensional complex vector space and let us choose a volume-form on $V$
$$
{\rm vol} :\wedge^6 V \to \CC
$$
and let us equip
$\wedge^3 V$ with the symplectic form
$(\alpha, \beta)_V := {\rm vol}(\alpha \wedge \beta)$.

Let ${\textbf{LG}}(\wedge^3 V)$ be the symplectic Grassmannian parametrizing Lagrangian subspaces of
$\wedge^3 V$. Given a non-zero $v \in V$ let
$$
F_v := \{ \alpha \in
\wedge^3 V | v \wedge \alpha = 0 \}
$$
be the sub-space of
$\wedge^3 V$ consisting of multiples of $v$. $( \ , \ )_V$ is zero on $F_v$ and $\dim (F_v) =10$, thus $F_v \in LG(\wedge^3 V)$. 
Let
\begin{equation}\label{eq:inclusion}
F \subset \wedge^3 V \otimes {\mathcal O}_{\PP(V)}
\end{equation}
be the sub-vector-bundle with fiber $F_v$ over $[v]\in \PP(V )$. 
Then
\begin{equation}\label{eq:det}
\det F \cong 
{\mathcal O}_{\PP(V )}(-6)
\end{equation}

Given $A \in {\mathbf LG}(\wedge^3 V)$ we let
$Y_A = \{[v] \in {\PP(V)} | F_v \cap A \ne \{0\} \}$.
Thus $Y_A$ is the degeneracy locus of the map
$\lambda_A:F \to
(\wedge^3 V/A) \otimes {\mathcal O}_{\PP(V )}$
where $\lambda_A$ is given by inclusion (\ref{eq:inclusion}) 
followed by the quotient map
\begin{equation}\label{eq:vb}
\wedge^3 V \otimes {\mathcal O}_{\PP(V)} \to (
\wedge^3 V/A) \otimes {\mathcal O}_{\PP(V)}
\end{equation}
Since the vector-bundles appearing in (\ref{eq:vb}) have equal rank the determinant of $\lambda_A$ makes sense and
$Y_A = V (\det \lambda_A)$; this formula shows that $Y_A$ has a natural structure of closed subscheme
of $\PP(V)$. 
By (\ref{eq:det}) we have $\det \lambda_A \in H^0({\mathcal O}_{\PP(V)}(6))$ 
and hence $Y_A$ is either a sextic hypersurface
or $\PP(V)$. 
An {\it EPW sextic} is a sextic hypersurface in $\PP^5$ which is projectively equivalent to $Y_A$ for
some $A \in {\mathbf{LG}}(\wedge^3 V)$, and a {\it double EPW sextic}
is its associated double covering studied by O'Grady.

\subsection{Necessary conditions and negative Pell's equations} 

Next we will look for necessary conditions to have a birational map between an Hilbert square of a K3 surface of degree $d$ and a double EPW sextic.
We will follow Mukai (\cite{Mu1}).


\begin{prop}\label{ccc}
Let ${\widetilde Y} \to Y$ be a double EPW sextic which is
smooth and birational to $S^{[2]}$ for 
a primitively polarized K3 surface $S$ of degree $d = 2g-2 \ge 10$ and Picard number $1$. 
Then     
the negative Pell's equation 
$$y^2 - (g-1)x^2 = -1$$ 
has an integer solution.
\end{prop}

\begin{proof}
By a result of Mukai (see \cite{Mu1} Corollary 5.9),  
if $Y$ is birational to $S^{[2]}$, then there exists an isometry 
between the Neron-Severi lattices $NS(Y) \cong NS(S^{[2]})$.
Recall that $NS(S^{[2]}) = \Z h + \Z \delta$, where $(h,h) = d = 2g-2$,
$(h,\delta) = 0$, and $(\delta,\delta) = -2$. 

Let $\pi: \widetilde{Y} \rightarrow Y$ be the double covering defined by 
the antisymplectic involution, as in \cite{OG1}, \cite{OG2}. 
The EPW polarization $\gamma$ on $\widetilde{Y}$ is the preimage 
of the hyperplane class on the EPW sextic $Y \subset \PP^5$. 
Therefore the intersection index 
$$\gamma^4 = \deg{\pi} \cdot deg(Y) = 2 \cdot 6 = 12.$$ 
Since the double EPW sextic $\widetilde{Y}$ is a deformation of a Hilbert square of 
a K3 surface (see \cite{OG2}) then the Fujiki constant $c(\widetilde{Y}) = c(S^{[2]}) = 3$,
see (1.0.1) and (4.1.4) in \cite{OG1}.
Therefore for the Beauville form $(.,.)$ on ${\rm NS}(\widetilde{Y})$ one 
will have:
$$12 = \gamma^4 = c(\widetilde{Y})(\gamma,\gamma)^2 = 3(\gamma,\gamma)^2,$$
which yields  
$$
(\gamma,\gamma) = 2.
$$

By the isometry $NS(\widetilde{Y}) \cong NS(S^{[2]})$ we can identify 
$\gamma$ with an element of $NS(S^{[2]})$, 
i.e. the birationality of $\widetilde{Y}$ with the Hilbert square of a K3 surface 
as above implies that there exist integers $x,y$ such that 
$\gamma = xh - y\delta$. 
Then 
$$
2 = (\gamma,\gamma) = (xh - y\delta , xh - y\delta) = dx^2 - 2y^2 = (2g-2)x^2 - 2y^2,
$$
from where 
$$y^2 - (g-1)x^2 = -1.
$$    
\end{proof}

\begin{rema} 
It is well known that if $p$ is prime then 
the negative Pell's equation $y^2 - px^2 = -1$ 
has a solution if and only if $p = 2$ or $p \equiv 1$ (mod 4),
see e.g. Theorem 3.4.2 in \cite{AAC}.  

Below we use the case when $p = 5$ which corresponds to 
a double EPW sextic birational to the Hilbert 
square of a K3 surface of degree 10, see \S 4.3 in \cite{OG1}. 
For $p = 5$, the minimal solution of $y^2 - 5x^2 = -1$ 
is $(y,x) = (2,1)$, 
and all solutions $(y_n,x_n)$, $n \ge 0$ to $y^2 - 5y^2 = -1$ 
are given by 
$$
2y_n = (1+2\sqrt{5})(2+\sqrt{5})^{2n}  + (1-2\sqrt{5})(2-\sqrt{5})^{2n}, 
$$
$$
2x_n = (2+1/\sqrt{5})(2+\sqrt{5})^{2n} + (2-1/\sqrt{5})(2-\sqrt{5})^{2n}, 
$$
see e.g. Theorem 3.4.1 on p.141 and the formulas on p.305 in \cite{AAC}.  
\end{rema}


\section{Double EPW sextics and Hilbert squares of K3 surfaces}\label{s:main}

\noindent We can now state main result of the paper, which is the following: 


\begin{theo}\label{X4}
The Hilbert square of a very general K3 surface of degree 
$d = d(n) = 2(4n^2+8n+5)$, $n \ge 1$  
is birational to a double EPW sextic $\widetilde{Y}$. 
\end{theo}

The proof of Theorem \ref{X4} uses similar methods used by Hassett in \cite{Ha1} to show a (stronger, in some sense) similar result for the variety of lines on a cubic fourfold.
Our main observation is that the same approach can be used also in the case of 
double EPW sextics. We divide the proof into several parts:

\subsection{The birational involution on $S^{[2]}$
for a K3 surface $S$ of degree 10}\label{birinv}

For a K3 surface $S$ with a polarization $f$ of degree 
$f^2 = d = 2g-2$ and Picard number $1$, any curve $C \in |f|$ defines a 
divisor $F_C = \{ \xi \in S^{[2]}: Supp(\xi) \cap C \not=\emptyset \}$
on $S^{[2]}$. All divisors $F_C$ belong to the same 
class $f \in {\rm NS}(S^{[2]})$. We use the same notation 
for the class $f$ and for the polarization $f$ on $S$. 
The class of the diagonal  
$\Delta = \{ \xi \in S^{[2]}: Supp(\xi) = {\rm point} \}$ 
is divisible by two in ${\rm NS}(S^{[2]}$), and if 
$\Delta = 2 \delta$ then  
$${\rm NS}(S^{[2]})  = \ZZ f + \ZZ \delta.$$ 

If $(.,.)$ is the Beauville form on ${\rm NS}(S^{[2]})$, 
then 
$$(f,f) = d, \ (f,\delta) = 0, \ (\delta,\delta) = -2.$$

If on $S$ there is a polarization $f$ of degree $d = 10$,
then there exists a birational involution 
$$
j: S^{[2]} \rightarrow S^{[2]}. 
$$ 
For the general pair $(x,y)$ of points on the general $S$ 
the involution $j$ can be described geometrically as follows 
(for more detail see \cite{OG1}):

Let $G = G(2,5) = G(1:\PP^4) \subset \PP^9$ be the grassmannian of 
lines in $\PP^4$.  
By \cite{Mu1}, the general smooth K3 surface $S$ of degree $10$ 
is a quadratic section $S = V_5 \cap Q$ 
of the unique smooth del Pezzo threefold $V_5 =  G \cap \PP^6$, 
which is a prime Fano threefold of index 2 and 
degree 5. 
By the general choice of $S \subset V_5$, the 
general non-ordered pair of points $(x,y)$ on $S \subset V_5$ 
is a general pair of points on $V_5$. The del Pezzo threefold $V_5$ has 
the property that through the general pair of 
points on $V_5$ passes a unique conic $q = q_{x,y}$. 
Indeed, let $l_x, l_y$ be the two lines in $\PP^4$ 
representing the points $x,y \in V_5 \subset G = G(1:\PP^4)$.
By the general choice of $x,y$, the lines $l_x$ and $l_y$ do not 
intersect each other and span a 3-space $\PP^3_{x,y} \subset \PP^4$.
Any conic $q \subset G$ which passes through $x$ and $y$ 
lies in the Pl\"ucker quadric $G(2,4)_{x,y} = G(1:\PP^3_{x,y}) \subset G$. 
In addition, since $V_5 = G \cap \PP^6$ then any conic on $V_5$ 
which passes through $x$ and $y$ lies on the codimension 3 subspace 
$\PP^6 \subset \PP^9 = Span(G)$. Therefore the set of conics on $V_5$ 
which pass through $x$ and $y$ sweep out the intersection 
$q_{x,y} = G(2,4)_{x,y} \cap \PP^6$, which by 
the general choice of $x,y$ is a codimension 3 linear section of
the 4-dimensional quadric 
$G(2,4)_{x,y}$, i.e. a conic. Since $S = V_5 \cap Q$ 
is a quadratic section of $V_5$, the conic $q_{x,y}$ intersects $S$
at $x,y$ and a pair of other 2 points $x',y'$. This defines a 
birational involution
$$
j: S^{[2]} \longrightarrow S^{[2]} \ \ j(x,y) = (x',y').
$$ 

Let  $r = f - 2\delta \in {\rm NS}(S^{[2]}) = \ZZ f + \ZZ \delta$.
Then 
$$
(r,r) = (f - 2\delta, f - 2\delta) = (f,f) + 4(\delta,\delta) = 2.
$$
By Propositions 4.1 and 4.21 in \cite{OG1}, on ${\rm NS}(S^{[2]}) = \ZZ f + \ZZ \delta$ 
the involution $j$ is given by the reflection with respect to $r$   
$$
j: z \mapsto j(z) = -z + (z,r)r = -z + (z,f - 2\delta)(f-2\delta). 
$$
We keep the same notation for the involution $j$ on $S^{[2]}$ 
and the involution $j$ on ${\rm NS}(S^{[2]})$.
In particular, 
$$
j(f) = -f + (f, f - 2\delta)(f-2\delta)  
          = -f + 10(f-2\delta) = 9f - 20\delta ,
$$
$$
j(\delta) = -\delta + (\delta, f - 2\delta)(f-2\delta)  
          = -\delta + 4(f-2\delta) = 4f - 9\delta.
$$

\subsection{The Hilbert square of a K3 surface of degree 10 as 
a double EPW sextic}\label{h}
Let $S \subset V_5 \subset G = G(1:\PP^4)$ be a very general K3 surface 
with a polarization $h$ of degree 10, 
where $V_5 = G \cap \PP^6$ is as above.
By \cite{OG1}, \cite{OG2}, the Hilbert square 
$S^{[2]}$ is a special case (as a birational equivalence class) of a double EPW sextic. 
The double covering 
is defined by the 
involution $j$ on $S^{[2]}$, and can be described 
as follows. 
 
Let $\PP^5 = |I_S(2)|$ be the projective space of quadrics in $\PP^6$ 
which contain $S \in |{\mathcal O}_{V_5}(2)|$. 
In $\PP^5$, the quadrics which contain $V_5$ form a hyperplane
identified with the space of Pfaffian quadrics.
Let $\xi \in S^{[2]}$, and let $\PP^1_{\xi} = Span(\xi)$.
Then  $\xi$ defines a hyperplane 
$$
\PP^4_{\xi} = |I_{S\cup \PP^1_{\xi}}(2)| \subset |I_S(2)| = \PP^5.
$$  
If $S$ does not contain lines, which is the general case, 
then the map 
$$
\pi: S^{[2]} \rightarrow \check{\PP}^5, \ \xi \mapsto \PP^4_{\xi} 
$$
is well defined for any $\xi \in S^{[2]}$. 
The map $\pi$ is (generically) the double covering defined by the 
involution $j$. We shall show only that the images of two 
involutive elements by $\pi$ coincide; 
for more detail see \cite{OG1} and \cite{Mu1}. 
Indeed, if $j(\xi)$ is the involutive 
of $\xi$, then the lines $\PP^1_{\xi} = Span(\xi)$
and $\PP^1_{j(\xi)} = Span(j(\xi))$ intersect each other, 
since by construction of $j(\xi)$,
$\xi + j(\xi)$ lie on a conic -- see above.
Since the lines $\PP^1_{\xi}$ and $\PP^1_{j(\xi)}$
are bisecant or tangent to $S$ and intersect each other, 
any quadric which contains $S$ together with one 
of these two lines contains also the other line. 
By the definition of $\pi$, the last yields that the images 
$\pi(\xi)$ and $\pi(j(\xi))$ coincide. 

By \S 4.3 in \cite{OG1}, 
the image $Y_0 \subset \check{\PP}^5$ of the double covering $\pi$ 
is an EPW sextic, defining the double EPW sextic   
$$
\widetilde{Y}_0 \rightarrow Y_0,
$$
which is birational to the Hilbert square $S^{[2]}$,
see also Theorem 4.15 in \cite{OG3}. 

\medskip

In the sequel we will also need the following result of O'Grady: 

\begin{lemm}\label{(F)} {\rm (see Proposition 4.21 and Corollary 5.21 in \cite{OG1})}. 
The class $\gamma$ of the EPW polarization 
$\pi^*({\mathcal O}_{Y_0}(1)) \in {\rm NS}(S^{[2]}) = \ZZ h + \ZZ\delta$ 
is $\gamma = h - 2\delta$.
\end{lemm} 

\begin{rema}
Here we assume that $NS(S) \cong \ZZ$ and denote by $h$ the ample generator.
The EPW-polarization $\gamma = xh-y\delta$
is $j$-invariant, i.e. $j(\gamma) = \gamma$, where  
$$
j: z \mapsto -z + (z,r)r 
$$
is the involution defined by $r = h - 2\delta$, 
which interchanges the two preimages of the 
general point $p \in Y_0$, see Subsection \ref{h}.
The equality $\gamma = j(\gamma) = -\gamma + (r,\gamma)r$ yields
$2\gamma = (r,\gamma)r$, i.e. $\gamma$ is proportional to $r = h-2\delta$.
Since $\gamma$ is primitive,
i.e. not divisible by an integer,
$\gamma = r$. 
\end{rema}

\subsection{K3 surfaces with two polarizations of degree 10}\label{fh}

\begin{lemm}\label{bbb}
Let $R$ be the rank two lattice $R = \Z f + \Z h$ with intersection form 
$$
\begin{array}{c|cc}
 & f & h \\
 \hline 
f & 10 & n+10 \\
h & n+10 & 10
\end{array} 
$$
where $n \ge 1$. 
Then there exists a K3 surface $S$ with $NS(S) = \Z f + \Z h$, 
such that $f$ and $h$ are two very ample polarizations on $S$. 
\end{lemm} 

\begin{proof}
Let $\Lambda = U^{\oplus 3} \oplus E_8(-1)^{\oplus 2}$ be the K3 cohomology lattice.
By Theorem 2.4 in \cite{LP}, 
there exists an embedding $R \subset \Lambda$. By the surjectivity of 
the period map for K3 surfaces one can assume that e.g. $f$ is a very ample polarization 
on a K3 surface $S$. Since $(f,h) > 0$ then the divisor class $h$ is effective,
and one needs to see that $h$ is very ample.    
If $h$ is not ample then on $S$ will exist a (-2)-curve $E$ such that $h \cdot E \le 0$. 
If then $k = h \cdot E = 0$ then $R_0 = \Z h + \Z E$ will be a sublattice of $R$ 
of discriminant $d(R_0) = -20$. Since $R_0 \subset R$ then $d(R) = -n(n+20)$ divides 
$d(R_0) = -20$, which is not possible. It remains the possibility when $hE = -k < 0$. 
Since $E^2 = -2$, then $E$ defines a reflection 
$$
r_E: x \mapsto \bar{x} = x + (x,E)E,
$$
$x \in NS(S) \supset R$.  
In particular, $\bar{h} = h - kE$, $(\bar{h},\bar{h}) = (h,h) = 10$,
and $(f,\bar{h}) = (f,h-kE) = (f,h) - k(f,E) < (f,h)$ since $f$ is (very) ample and $E$ 
is effective.   
Since $\bar{h} \in R$ then $R' = \Z f + \Z \bar{h}$ is a sublattice of $R = \Z f + \Z h$. 
Therefore $d(R)$ divides $d(R')$, and since both $d(R)$ and $d(R')$ are negative, 
then $d(R') \le d(R)$. 
But 
$$
d(R') = (f,f)(\bar{h},\bar{h}) - (f,\bar{h})^2 =
$$
$$
= (f,f)(h,h) -  (f,\bar{h})^2
> (f,f)(h,h) - (f,h)^2 = d(R),
$$ 
contradiction. This proves the Lemma. 
For more detail see Lemma 4.3.3 and \S6 in \cite{Ha1}.
\end{proof}

\subsection{Proof of Theorem \ref{X4}.}


Let $S$ be a very general K3 surface with a primitive polarization 
$h$ of degree 10 as in \ref{h}.     
Denote by $\widetilde{Y}_0$ the corresponding double EPW sextic to 
$S^{[2]}$. 
Let $\widetilde{Y}_t$ be a local deformation of $\widetilde{Y}_0$
in the polarization $\gamma = h - 2\delta$ as a double EPW sextic
$\pi_t: \widetilde{Y}_t \rightarrow Y_t.$
Since $\widetilde{Y}_t$ is a deformation of a Hilbert square of 
a K3 surface, the Fujiki constant $c(\widetilde{Y}_t) = c(S^{[2]}) = 3$, 
and as in the proof of Proposition \ref{ccc},
we get $(\gamma,\gamma) = 2$.

Let $S$ be a very general K3 surface with two polarizations 
$f$ and $h$ (generating the Neron-Severi lattice) as in Lemma \ref{bbb}. 
By above, e.g. in the polarization $h$, 
the Hilbert square $S^{[2]}$ is birational to 
a double EPW sextic $\widetilde{Y}_0$. 
By Proposition 2.2 and Theorem 4.15 in \cite{OG3}, 
${Y}_0$ belongs to the locus $\Delta - \Sigma$
(ibid. (0.0.7)-(0.0.8)), 
and by Proposition 6.3 of \cite{OG2} has a unique singular point $p_0$ of multiplicity three.
The Hilbert square $S^{[2]} \rightarrow \widetilde{Y}_0$ 
is a small resolution of $p_0$ which is a contraction of 
a Lagrangian plane on $S^{[2]}$ to the point $p_0$. 

Next, we proceed as in the proof of Theorem 6.1.4 in \cite{Ha1}
for families of lines on cubic fourfolds, 
adapted to the case of double EPW sextics.  

By 
\cite{OG3} the period map for double EPW sextics extends regularly around 
the period point of $S^{[2]}$. 
Let ${\mathcal F}_{g(n)}$ be the moduli space of primitively polarized K3 surfaces 
$S'$ of genus $g(n) = d(n)/2+1$. 
By the surjectivity of the period map for K3 surfaces 
(see \cite{LP}),  
one can consider $h_2=\gamma+(2n+2)\delta_2 \in \Pi$ as the (quasi) polarization of a K3 surface 
of genus $g(n)$, with $\delta_2=4f-9\delta \in \Pi$ the class of the 
half-diagonal on its Hilbert square, 
see also the proof of Proposition 7 in \cite{BHT}. 
By Proposition 10, Theorem 6 and Remark 2 on p. 779-780 of \cite{Be}
(see also Theorem 6.1.2 in \cite{Ha1}) 
in the 20-dimensional local moduli space ${\mathcal M}$ of double EPW sextics 
$\widetilde{Y}_t$ around $\widetilde{Y}_0$ 
the condition that $\delta_2=4f-9\delta$ remains algebraic, 
i.e. an element of $NS(\widetilde{Y}_t)$, 
describes locally a smooth component of the divisor 
in ${\mathcal M}$ on which $\widetilde{Y}_t$ remains birational to a  
Hilbert square of a K3 surface $S_t$ of genus $g(n)$.

For the general double EPW sextic $\widetilde{Y}_t$ 
as above, the lattice 
${\rm NS}(\widetilde{Y}_t)$ has rank two, and
is the saturation of the rank two sublattice 
$$
\Pi = \ZZ \gamma + \ZZ\delta_2 = \ZZ(h-2\delta) + \ZZ(4f-9\delta).  
$$
Since $\Pi$ is saturated, then ${\rm NS}(\widetilde{Y}_t)$ coincides with $\Pi$,
in particular the discriminant $d({\rm NS}(\widetilde{Y}_t))$ is the 
discriminant of $\Pi$. 
By using $(\gamma,\gamma) = 2$ and $(\delta_2,\delta_2) = -2$, 
and the intersection table from Lemma \ref{bbb}, 
we compute  
$$
(\gamma,\delta_2) = (h - 2\delta, 4f - 9\delta) 
= 4(h,f) + 18(\delta,\delta) = 4(n+10) - 36 = 4n + 4. 
$$
Therefore 

$$
d({\rm NS}(\widetilde{Y}_t))  = d(\Pi)  
= 
 det
\left( \begin{array}{cc}
(\gamma,\gamma)    & (\gamma,\delta_2) \\ 
(\gamma,\delta_2)  & (\delta_2,\delta_2) 
\end{array} \right) 
= 
$$
$$
=
(\gamma,\gamma)(\delta_2,\delta_2) - (\gamma,\delta_2)^2
                = 2(-2) - (4n+4)^2 =
$$
$$ 
                = (-2)(8n^2+16n+10).
$$
                
Therefore $\widetilde{Y}_t$ is birational to the Hilbert square of a K3 
surface $S_t$ of degree 
$$
d(n) = 2g(n)-2 = 8n^2+16n+10 = 2(4(n+1)^2 + 1).
$$

This proves Theorem \ref{X4}. \qed

\medskip
 
\begin{rema}
Let $NS(S_t^{[2]}) = \Z h_2 + \Z \delta_2$, where $h_2$ is the primitive 
polarization class on $S_t^{[2]}$. By using that 
$$
NS(S_t^{[2]}) \cong NS(\widetilde{Y}_t) \cong \Pi,
$$ 
we can compute directly the degree 
$d(n) = (h_2,h_2)$
of the K3 surface $S_t$. 
Since $h_2$ is primitive and orthogonal to the half-diagonal class 
$\delta_2$, and since 
$$
\Pi \cap \delta_2^{\perp} = \Z (\gamma + (2n+2)\delta_2),  
$$ 
then $h_2 = \gamma + (2n+2)\delta_2$. From here, 
and by the intersection table from Lemma \ref{bbb}, we get again 
$$
d(n) = (h_2,h_2) = (\gamma + (2n+2)\delta_2, \gamma + (2n+2)\delta_2) = 
$$
$$
= (\gamma,\gamma) + 2(2n+2)(\gamma,\delta_2) + (2n+2)^2(\delta_2,\delta_2) = 
$$
$$          
= 2 + 2(2n+2)(4n+4) + (2n+2)^2 (-2) =
$$
$$
= 2 + 8(n+1)^2 = 8n^2+16n+10 .
$$
The intersection matrix of $\Pi$ in the base $h_2,\delta_2$, 
is
$$
\begin{array}{c|cc}
 & h_2   & \delta_2 \\
 \hline 
h_2      & d(n)  & 0 \\
\delta_2 & 0            & -2.
\end{array} 
$$
\end{rema}

\begin{rema}
Theorem \ref{X4} implies that on the Hilbert square $S^{[2]}$ of a general 
K3 surface $S$ of degree $d(n) = 8n^2 +16n + 10$, $n \ge 1$ the EPW polarization 
$\gamma = h_2 - (2n+2)\delta_2$ defines an antisymplectic birational involution. 

This proves the O'Grady conjecture that on the Hilbert square of a K3 surface 
of genus $g = r^2 + 2, r \ge 0$ there exists an antisymplectic involution 
(see (4.3.3) in \cite{OG1}), in the case when $r$ is even.
Indeed, for $d(n) = 8n^2 + 16 n + 10 = 2(4(n+1)^2+1)$, 
the genus of $S$ is $g(n)$ = $d(n)/2 + 1 = 4n^2 + 8n + 6$ = $(2n+2)^2 + 2$, 
and for $n \ge 1$, \ $r = 2n+2$ covers all even numbers $r \ge 4$. The 
case $r = 0$ is well known, and the case $r = 2$ is studied in detail by O'Grady, 
see e.g. \S 4.3 in \cite{OG1}. The odd case $r = 1$ is studied in \cite{De} 
and \cite{Fe}.  
\end{rema}



\vspace{0.6cm}

{\footnotesize 

\noindent
{\sc Atanas Iliev}: 
Seoul National University,
Department of Mathematics,
Gwanak Campus, Bldg.27, 
Seoul 151-747, Korea; 
{\sc e-mail}: ailiev@snu.ac.kr

\medskip
 
\noindent
{\sc Carlo Madonna}:
Aut\'onoma University of Madrid,
Faculty of Teacher Training and Education,
C/Fco. Tom\'as y Valiente 3,
Madrid E-28049, Spain; 
{\sc e-mail}: carlo.madonna@uam.es


\begin{thebibliography}{AAAA} 

\bibitem[AAC]{AAC} 
Andreescu T., Andrica D., Cucurezeanu I., 
An Introduction to Diophantine Equations, 
Birkh\"auser, 2010.


\bibitem[Be]{Be} 
Beauville A.,
{\it Vari\'et\'es K\"ahleriennes dont la premi\`ere classe de Chern est nulle}, 
J. Differ. Geom., {\bf 18} (1983), 755--782. 


\bibitem[BD]{BD}
Beauville A., Donagi R., 
{\it La vari\'{e}t\'{e} des droites d'une hypersurface cubique de dimension 4}, 
{\it C. R. Acad. Sci. Paris S\'er. I Math.}  {\bf 301}  (1985),  703--706.


\bibitem[BHT]{BHT}
Bayer A., Hassett B., Tschinkel Yu., 
{\it Mori cones of holomorphic symplectic varieties of K3 type},
arxiv:1307.2291.


\bibitem[De]{De}
Debarre O., 
{\it Un contre-exemple au th\'eor\`eme de Torelli pour les vari\'ete\'s symplectiques irr\'eductibles},  
C. R. Acad. Sci. Paris {\bf 299} (1984), 681--684. 


\bibitem[DIM1]{DIM1} 
Debarre O., Iliev A., Manivel L.,
{\it Special prime Fano fourfolds of degree 10 and index 2},
arxiv:1302.1398, to appear.



\bibitem[EPW]{EPW}
D. Eisenbud, S. Popescu, Ch. Walter, 
Enriques surfaces and other non-Pfaffian
subcanonical subschemes of codimension 3, Comm. Algebra 28 (2000), no. 12, 5629--5653.


\bibitem[Fe]{Fe}
Ferretti A., 
{\it Special subvarieties of EPW sextics}, 
Math. Zeitschrift {\bf 272} (2012),1137--1164.


\bibitem[Gus]{Gus}	Gushel' N.,
{\it Fano varieties of genus $6$},  
Math. USSR-Izv. {\bf 21} (1983), 445--459.


\bibitem[Ha1]{Ha1}
Hassett  B., 
{\it Special cubic fourfolds}, 
Compositio Math.  {\bf 120} (2000),  1--23.
 
 
 



\bibitem[LP]{LP}
Looijenga E., Peters C., 
{\it Torelli theorems for K\"ahler K3 surfaces},
Compositio Math., {\bf 42} No.2 (1980), 145--186.


\bibitem[Mu1]{Mu1}    
Mukai, S.,  
{\it Moduli of vector bundles on $K3$ surfaces, and symplectic manifolds},
Sugaku Expositions, Vol.{\bf 1}, No.2 (1988), 139--174.








\bibitem[OG1]{OG1} 
O'Grady K., 
{\it Involutions and linear systems on holomorphic symplectic manifolds}, 
Geom. Funct. Anal. {\bf 15} (2005), No.6, 1223--1274.


\bibitem[OG2]{OG2} 
O'Grady K., 
{\it Irreducible symplectic 4-folds and Eisenbud-Popescu-Walter sextics},
Duke Math. J. {\bf 134} (2006), No. 1, 99--137.


\bibitem[OG3]{OG3} 
O'Grady K., 
{\it Double covers of EPW-sextics}, 
Michingan Math. J., Vol. {\bf 62}, issue 1 (2013), 143--184.


\bibitem[OG4]{OG4} 
O'Grady K., 
{\it EPW-sextics: taxonomy}, 
Manuscripta Math. {\bf 138} (2012), No.1, 221--272.

\bibitem[OG5]{OG5} 
O'Grady K.,
{\it Dual Double EPW-sextics and Their Periods}, 
Pure and Appl. Math. Quart., Vol.{\bf 4}, No.2 (2008), 427--468.


\end{thebibliography}
\end{document}